\documentclass[a4paper,10pt]{amsart}
\usepackage[utf8]{inputenc}
\usepackage{amsmath,amsthm,amssymb,mathrsfs}

\usepackage{dsfont}
\newtheoremstyle{mystyle}{}{}{}{}{\bf}{}{\newline}{}
\theoremstyle{mystyle}

\usepackage{enumitem}
\usepackage[toc,page]{appendix}
\usepackage{mathtools,array}
\usepackage{graphicx,lipsum}
\usepackage{fancyhdr,floatpag}
\floatpagestyle{fancy}
\usepackage[usenames,dvipsnames]{xcolor}

\usepackage{hyperref}
\hypersetup{linkbordercolor=NavyBlue,linkcolor=NavyBlue}
\hypersetup{citebordercolor=NavyBlue,citecolor=NavyBlue}

\usepackage{upgreek}
\usepackage{mathtools}

\usepackage{frcursive}

\usepackage[style=numeric,backend=bibtex,maxnames=5]{biblatex}
\renewbibmacro{in:}{}
\usepackage{blkarray}
\usepackage{multirow}

\numberwithin{equation}{section}
\theoremstyle{plain}
\newtheorem{theorem}{Theorem}
\newtheorem{thm}{Theorem}[section]
\newtheorem{cor}[thm]{Corollary}
\newtheorem{dfn}[thm]{Definition}
\newtheorem{lem}[thm]{Lemma}
\newtheorem{pps}[thm]{Proposition}

\theoremstyle{remark}
\newtheorem{rmk}[thm]{Remark}
\newtheorem{exm}[thm]{Example}

\newcommand{\Z}{\mathds{Z}}
\newcommand{\C}{\mathds{C}}
\newcommand{\Q}{\mathds{Q}}

\newcommand{\N}{\mathds{N}}
\newcommand{\R}{\mathds{R}}
\newcommand{\G}{\mathbf{G}}
\newcommand{\m}{\mathbf{m}}
\newcommand{\p}{\mathfrak{p}}
\newcommand{\Pp}{\mathfrak{P}}
\newcommand{\ri}{\mathcal{O}}
\newcommand{\lri}{\mathfrak{o}}
\newcommand{\Lri}{\mathfrak{O}}
\newcommand{\g}{\mathfrak{g}}

\newcommand{\lrip}{\lri/\p^N}
\newcommand{\ric}{\ri_{\p}}
\newcommand{\op}{\lri/\p}
\newcommand{\W}{W^{\lri}}

\newcommand{\z}{\mathfrak{z}}

\newcommand{\pint}{\mathscr{Z}}
\newcommand{\cczeta}{\zeta^\textup{cc}}
\newcommand{\rzeta}{\zeta^\textup{irr}}
\newcommand{\rtzeta}{\zeta^{\widetilde{\textup{irr}}}}
\newcommand{\kzeta}{\zeta^\textup{k}}
\newcommand{\rg}{\widetilde{r}}
\newcommand{\bzf}{\mathcal{Z}}
\newcommand{\brzf}{\mathcal{Z}^\textup{irr}}
\newcommand{\bcczf}{\mathcal{Z}^\textup{cc}}
\newcommand{\bzfs}{\mathcal{Z}^{\ast}}
\newcommand{\bzfq}{\mathcal{Z}_{\G(\ri)}^{\ast,\mathcal{Q}}(s_1,s_2)}
\newcommand{\Zip}{\widetilde{\mathcal{Z}^{\ast}_{i,\p}}(s_1,s_2)}
\newcommand{\Zips}{\widetilde{\mathcal{Z}^{\ast}_{i,\p}}}
\newcommand{\Zipb}{\overline{\mathcal{Z}^{\ast}_{i,\p}}(s_1,s_2)}
\newcommand{\Zipbs}{\overline{\mathcal{Z}^{\ast}_{i,\p}}}
\newcommand{\rrel}{\widetilde{\brzf_{\G(\lri)}}(s_1,s_2)}
\newcommand{\crel}{\widetilde{\bcczf_{\G(\lri)}}(s_1,s_2)}
\newcommand{\arel}{\widetilde{\bzfs_{\G(\lri)}}(s_1,s_2)}

\newcommand{\astbc}{\mathscr{B}_{\G(\ri)}^{\ast}(s_1,s_2)}
\newcommand{\astgc}{\mathscr{G}_{\G(\ri)}^{\ast}(s_1,s_2)}
\newcommand{\astzeta}{\zeta^{\ast}}

\newcommand{\rk}{\textup{rk}}

\newcommand{\re}{\text{Re}}
\newcommand{\im}{\text{Im}}

\newcommand{\ua}{u_{A}}
\newcommand{\ub}{u_{B}}

\newcommand{\cc}{{c}}

\newcommand{\irrup}{\textup{irr}}
\newcommand{\ccup}{\textup{cc}}
\newcommand{\asta}{\mathbf{a}^{\ast}}
\newcommand{\astb}{\mathbf{b}^{\ast}}
\newcommand{\astA}{A^{\ast}}
\newcommand{\astB}{B^{\ast}}
\newcommand{\poweri}{-A_{1i}^{\ast}s_1-A_{2i}^{\ast}s_2-B_{i}^{\ast}}
\newcommand{\powerj}{-A_{1j}^{\ast}s_1-A_{2j}^{\ast}s_2-B_{j}^{\ast}}
\newcommand{\dci}{A_{1i}^{\ast}s_1+A_{2i}^{\ast}s_2}
\newcommand{\dcj}{A_{1j}^{\ast}s_1+A_{2j}^{\ast}s_2}
\newcommand{\setopt}{(s_1,s_2)\in\C^2}
\newcommand{\lfct}{\mathcal{L}}
\newcommand{\prodq}{\prod_{\p \notin Q}}
\newcommand{\sumq}{\sum_{\p \notin Q}}

\newcommand{\prodr}{\prod_{i \in \Rs}}
\newcommand{\sumr}{\sum_{i \in \Rs}}

\newcommand{\sumw}{\sum_{i \in W'}}
\newcommand{\hol}{\mathscr{D}_{\G(\ri)}^{\ast}}
\newcommand{\holi}{\mathscr{D}_{\G}^{\ast}}
\newcommand{\mer}{\mathscr{M}_{\G(\ri)}^{\ast}}
\newcommand{\meri}{\mathscr{M}_{\G}^{\ast}}
\newcommand{\tmer}{\mathscr{M}_{\mathscr{G}^{\ast}_{\G(\ri)}}}
\newcommand{\mero}{\mathscr{M}_{\mathscr{G}^{\ast}_{\G(\ri)}}^1}
\newcommand{\mert}{\mathscr{M}_{\mathscr{G}^{\ast}_{\G(\ri)}}^2}
\newcommand{\D}{\mathscr{D}}
\newcommand{\Ps}{\textup{P}^{\ast}}

\newcommand{\Rs}{\mathscr{R}}
\newcommand{\Vp}{V_{\p}^{\ast}(s_1,s_2)}
\newcommand{\Vps}{V_{\p}^{\ast}}
\newcommand{\Vi}{V_{i}^{\ast}(s_1,s_2)}
\newcommand{\Vit}{\widetilde{V_{i}^{\ast}}(s_1,s_2)}
\newcommand{\Gf}{\mathscr{G}_{\G(\ri)}^{\ast}(s_1,s_2)}

\newcommand{\specp}{\textup{Spec}(\ri)\setminus\{(0)\}}
\makeatletter
\newcommand{\thickhline}{%
    \noalign {\ifnum 0=`}\fi \hrule height 1pt
    \futurelet \reserved@a \@xhline
}
\newcolumntype{"}{@{\hskip\tabcolsep\vrule width 1pt\hskip\tabcolsep}}
\makeatother

\makeatletter
\def\moverlay{\mathpalette\mov@rlay}
\def\mov@rlay#1#2{\leavevmode\vtop{%
   \baselineskip\z@skip \lineskiplimit-\maxdimen
   \ialign{\hfil$\m@th#1##$\hfil\cr#2\crcr}}}
\newcommand{\charfusion}[3][\mathord]{
    #1{\ifx#1\mathop\vphantom{#2}\fi
        \mathpalette\mov@rlay{#2\cr#3}
      }
    \ifx#1\mathop\expandafter\displaylimits\fi}
\makeatother

\newcommand{\bigcupdot}{\charfusion[\mathop]{\bigcup}{\cdot}}

\makeatletter
\renewcommand*\env@matrix[1][*\c@MaxMatrixCols c]{%
  \hskip -\arraycolsep
  \let\@ifnextchar\new@ifnextchar
  \array{#1}}
\makeatother

\DeclareMathOperator{\Mat}{Mat}

\DeclareMathOperator{\class}{k}

\DeclareMathOperator{\Gl}{GL}

\usepackage{titlesec}
\titleformat{\section}
       {\normalfont\bfseries}{\thesection}{1em}{}

\titleformat{\subsection}
       {\normalfont\bfseries}{\thesubsection}{1em}{}
			
\titleformat{\subsubsection}
       {\normalfont\bfseries}{\thesubsubsection}{1em}{}

\titleformat{\paragraph}
{\normalfont\normalsize\bfseries}{\theparagraph}{1em}{}
\titlespacing*{\paragraph}
{0cm}{2.25ex plus 1ex minus.2ex}{1ex plus .2ex}

\usepackage{etoolbox}
\newcommand{\zerodisplayskips}{%
  \setlength{\abovedisplayskip}{1pt}%
  \setlength{\belowdisplayskip}{1pt}%
  \setlength{\abovedisplayshortskip}{1pt}%
  \setlength{\belowdisplayshortskip}{1pt}}
\appto{\normalsize}{\zerodisplayskips}
\appto{\small}{\zerodisplayskips}
\appto{\footnotesize}{\zerodisplayskips}

\title[paper]{Analytic properties of bivariate representation and conjugacy class zeta functions of finitely generated nilpotent groups\vspace{-3ex}}
\author{Paula Lins}
\address{Fakult\"at f\"ur Mathematik -- Universit\"at Bielefeld, Germany}
\email{lins@math.uni-bielefeld.de}
\date{\today}

\begin{document}
\begin{abstract}
Let~$\G$ be a unipotent group scheme defined in terms of a nilpotent Lie lattice over the ring~$\ri$ of integers of a number field. 
We consider bivariate zeta functions of groups of the form $\G(\ri)$ encoding, respectively, the numbers of isomorphism classes 
of irreducible complex representations of finite dimensions and the numbers of conjugacy classes of congruence quotients of the associated groups. 
We show that the domains of convergence and meromorphic continuation of these zeta functions of groups $\G(\ri)$ are independent of the number field~$\ri$ considered, up to finitely 
many local factors.
\end{abstract}
\maketitle 
\vspace{-0.5cm}
\section{Introduction and statement of main results}
\thispagestyle{empty}
In \cite{PL18} we introduced two bivariate zeta functions which generalise the (univariate) zeta functions counting the following data of a group~$G$: 
\begin{align*}
  r_n(G)&=|\{\text{isomorphism classes of $n$-dimensional irreducible complex}\\
  &\phantom{r(G)}\text{representations of }G\}|,\\
  c_n(G)&=|\{\text{conjugacy classes of $G$ of cardinality }n\}|.
\end{align*}
In the context of topological groups, we only consider continuous representations.
If all numbers $r_n(G)$, respectively, $c_n(G)$ are finite---for instance if~$G$ is a finite group---one can define the following zeta functions encoding these data. 
\begin{dfn}\label{univa} Let~$s$ be a complex variable. The \emph{representation} and the \emph{conjugacy class zeta functions} of the group~$G$ are, respectively, 
\begin{align*}
  \rzeta_{G}(s)=\sum_{n=1}^{\infty}r_n(G)n^{-s} \vspace{0.7cm} \text{ and }\hspace{0.5cm}
  \cczeta_{G}(s)=\sum_{n=1}^{\infty}\cc_n(G)n^{-s}.
\end{align*}
\end{dfn}
We are, however, interested in finitely generated, torsion-free nilpotent groups ($\mathcal{T}$-groups for short) and the data above is not finite for such groups. 
We consider bivariate zeta functions encoding these data for the congruence quotients of the $\mathcal{T}$-group considered; 
one of the variables concerns the congruence quotients, while the other variable concerns the respective data
of these quotients.

Before we recall these bivariate zeta functions, we explain the groups of interest in this work, which are groups obtained from nilpotent Lie lattices, 
constructed as follows.
Denote by~$\ri$ the ring of integers of a number field~$K$. 
An \emph{$\ri$-Lie lattice} is a free and finitely generated $\ri$-module~$\Lambda$ together with an antisymmetric bi-additive form $[~,~]$ which satisfies the Jacobi identity. 
Denote by $\Lambda'=[\Lambda,\Lambda]$ the derived Lie sublattice of~$\Lambda$. 

For each $\ri$-algebra~$R$, set $\Lambda(R):=\Lambda \otimes_{\ri} R$. 
If $\Lambda$ has nilpotency class~$c$ and satisfies $\Lambda' \subseteq c!\Lambda$, the Campbell-Baker-Hausdorff formula defines a group operation~$\ast$ in $\Lambda(R)$,
so that the group $(\Lambda(R), \ast)$ is nilpotent with same nilpotency class~$c$ as~$\Lambda$.
This defines a unipotent group scheme $\G=\G_{\Lambda}$ over~$\ri$; see~\cite[Section~2.1.2]{StVo14}. 
In case of nilpotency class~$2$, there is a different construction of such~$\G$ which does not require $\Lambda' \subseteq 2\Lambda$, but in case this condition is satisfied, 
the unipotent group schemes obtained via this construction coincide with the latter ones; see~\cite[Section~2.4.1]{StVo14}.

The group $\G(\ri)$ is a $\mathcal{T}$-group of nilpotency class~$c$.
If~$R$ is a finitely generated pro-$p$ ring, for instance the completion~$\ric$ of~$\ri$ at some nonzero prime ideal~$\p$, then the group $\G(R)$ is a finitely 
generated nilpotent pro-$p$ group of nilpotency class~$c$.


\begin{dfn}\label{biva} The \emph{bivariate representation} and the \emph{bivariate conjugacy class zeta functions} of $\G(\ri)$ are, respectively,
\begin{align*} 
\brzf_{\G(\ri)}(s_1,s_2)&=\sum_{(0) \neq I \unlhd \ri }\rzeta_{\G(\ri/I)}(s_1)|\ri:I|^{-s_2}\text{ and}\\
\bcczf_{\G(\ri)}(s_1,s_2)&=\sum_{(0) \neq I \unlhd \ri }\cczeta_{\G(\ri/I)}(s_1)|\ri:I|^{-s_2},
\end{align*}
where $s_1$ and $s_2$ are complex variables. 
\end{dfn}
All $\mathcal{T}$-groups are virtually of the form $\G(\Z)$, as explained in~\cite[Section~5]{DuVo14} and \cite[Remark~2.4]{StVo14}. In~\cite[Remark~3.1]{PL18}, we give 
a possible definition of bivariate representation and bivariate conjugacy class zeta functions of abstract $\mathcal{T}$-groups in terms of finite-index subgroups of 
the form~$\G(\Z)$. 
Although this definition does depend on the chosen subgroup, two distinct subgroups yield bivariate zeta functions which coincide for all but finitely many local factors.
By `local factor' we mean the functions 
\begin{align}\label{localfactors}\bzfs_{\G(\ric)}(s_1,s_2)= \sum_{N=0}^{\infty}\astzeta_{\G(\lri/\p^N)}(s_1)|\lri:\p|^{-Ns_2},\end{align}
where $\ast\in\{\irrup,\ccup\}$ and $\ric$~denotes the completion of~$\ri$ at the nonzero prime ideal~$\p$. 
According to~\cite[Proposition~2.1]{PL18}, 
\begin{equation}\label{Euler}\bzfs_{\G(\ri)}(s_1,s_2)=\prod_{\p \in \specp} \bzfs_{\G(\ric)}(s_1,s_2).\end{equation}

The zeta functions of Definition~\ref{biva} are not defined only formally, they converge for~$s_1$ and~$s_2$ with sufficiently large real part; cf.~\cite[Proposition~2.4]{PL18}.

Our main result concerns the domain of convergence and meromorphy of these bivariate zeta functions; see Section~\ref{holmer} for definitions. 
We show that there exists a finite set~$\mathcal{Q}$ of prime ideals of~$\ri$ 
such that, for each $\ast\in\{\irrup,\ccup\}$, the domains of convergence and meromorphy of the bivariate function
\begin{equation}\label{mostprimes}
\bzfq=\prod_{\p \notin \mathcal{Q}} \bzfs_{\G(\ric)}(s_1,s_2),
\end{equation}
are independent of the ring of integers~$\ri$. This means that, for each finite extension $L/K$ with ring of integers $\ri_{L}$, the domains of convergence and meromorphic 
continuation of $\bzfs_{\G(\ri_L)}(s_1,s_2)$, up to finitely many local factors, are the same as the ones of $\bzfq$.
\begin{theorem}\label{thmA} Denote by $\hol$ the domain of convergence of $\bzfq$. This function admits meromorphic continuation to a domain $\mer \subset \C^2$ 
and for each finite extension $L/K$ with ring of integers~$\ri_{L}$ there exists a finite subset 
$\mathcal{Q}_{L} \subset \textup{Spec}(\ri_{L})$ such that the bivariate function 
\[\bzf_{\G(\ri_{L})}^{\ast,\mathcal{Q}_L}(s_1,s_2)=\prod_{\Pp \notin \mathcal{Q}_{L}} \bzfs_{\G(\ri_{L,\Pp})}(s_1,s_2),\]
where $\ri_{L,\Pp}$ is the completion of $\ri_{L}$ at the nonzero prime ideal $\Pp$, satisfies:
 \begin{enumerate} 
  \item \label{parti} The domain convergence of $\bzf_{\G(\ri_{L})}^{\ast,\mathcal{Q}}(s_1,s_2)$ coincides with~$\hol$ and  
  \item \label{partii} $\bzf_{\G(\ri_{L})}^{\ast,\mathcal{Q}}(s_1,s_2)$ admits meromorphic continuation to~$\mer$.  
  \end{enumerate}
\end{theorem}
In particular, the domains~$\hol$ and~$\mer$ are independent of~$\ri$. We hence write simply~$\holi=\hol$ and~$\meri=\mer$. 
The proof of Theorem~\ref{thmA} may be found in Section~\ref{pthmA}
\subsection{Class number zeta functions}
The \emph{class  number} of a group~$G$ is the total number~$\class(G)$ of its conjugacy classes.
If~$G$ is finite, the number~$\class(G)$ coincides with the total number of its irreducible complex characters, and hence $\class(G)=\cczeta_G(0)=\rzeta_G(0)$. 

According to~\cite[Section~1.2]{PL18}, the bivariate zeta functions of Definition~\ref{biva} associated to~$\G(\ri)$ specialise to the (univariate) 
\emph{class number zeta function} of~$\G(\ri)$, that is, the generating function 
\begin{equation}\label{specialisation}\kzeta_{\G(\ri)}(s):=\sum_{\{0\} \neq I \unlhd \ri}\class(\G(\ri/I))|\ri:I|^{-s}=\brzf_{\G(\ri)}(0,s)=\bcczf_{\G(\ri)}(0,s),\end{equation}
where $s$ is a complex variable. This specialisation applied to Theorem~\ref{thmA} yields the following. 
\begin{cor}
 Up to finitely many local factors, the class number zeta function of~$\G(\ri)$ converges on a domain~$\mathcal{D}_{\G,\class}^{\ast}\subseteq \C$, and admits meromorphic
 continuation to a larger domain~$\mathcal{M}_{\G,\class}^{\ast}$ with both~$\mathcal{D}_{\G,\class}^{\ast}$ and~$\mathcal{M}_{\G,\class}^{\ast}$ independent of~$\ri$.
\end{cor}
The term `conjugacy class zeta function' is sometimes used for what we call `class number zeta function'; see for instance~\cite{BDOP,Ro17, ask2, duSau05}.
\subsection{Related research}
Domains of convergence of zeta functions of groups $\G(\ri)$ being independent of the ring of integers $\ri$ is not a general property. 
The normal subgroup zeta function of the Heisenberg group~$\mathbf{H}(\ri)$ of the $3\times 3$-unitriangular matrices of~$\ri$, for instance, has abscissa of convergence depending 
on the degree of the extension $|K:\Q|$; see~\cite[Theorem~1.2]{ScVoI} and \cite[Theorems~3.2 and~3.8]{ScVoII}.

In~\cite[Corollary~1.3]{StVo14}, Stasinski and Voll show that the twist representation zeta functions of three infinite families of nilpotent groups of the form~$\G(\ri)$ 
of class~$2$ generalising the Heisenberg group $\mathbf{H}(\ri)$ of the $3\times 3$-unitriangular matrices of~$\ri$ have domains of convergence which are independent of~$\ri$ and 
admit meromorphic continuation to the whole~$\C$.
This is the zeta function counting the numbers $\rg_n(\G(\ri))$ of irreducible complex characters of dimension $n$ of $\G(\ri)$ up to tensoring by one-dimensional characters.
That is,
 \[\rtzeta_{\G(\ri)}(s)=\sum_{n=1}^{\infty}\rg_n(G)n^{-s},\] where $s$ is a complex variable.

Based on this corollary, Stasinski and Voll asked whether the abscissae of convergence of the twist representation zeta functions of all groups of 
the form~$\G(\ri)$ are independent of~$\ri$ and whether they admit meromorphic continuations; cf.~\cite[Question~1.4]{StVo14}.
This question was answered by Dung and Voll~\cite{DuVo14}. In~\cite[Theorem~A]{DuVo14}, they show that the twist representation zeta function of~$\G(\ri)$ has rational 
abscissa of convergence $\alpha(\G)$, which is independent of~$\ri$, and can be meromorphically continued to a domain 
$\{s \in \C\mid \re(s)>\alpha(\G)-\delta(\G)\}$, for some $\delta(\G) \in \Q_{>0}$ depending only on~$\G$. 
They also showed that, for each (abstract) $\mathcal{T}$-group $G$, the abscissa of convergence~$\alpha(G)$ of its twist representation zeta function is rational 
and that this zeta function admits meromorphic continuation to a half-plane of the form $\{s\in\C \mid \re(s)>\alpha(G)-\delta\}$, for some $\delta>0$; 
see~\cite[Theorem~B]{DuVo14}.

We remark that in case of nilpotency class~$2$, bivariate representation zeta functions specialise to twist representation zeta functions; see~\cite[Section~4.3]{PL18}. 

In~\cite[Theorems~1.4 and~1.6]{PL18II}, formulae are given for the bivariate representation and for the bivariate conjugacy class zeta functions of the three infinite families 
investigated in~\cite{StVo14}. 
These formulae show that the domains of convergence and meromorphy of these zeta functions are independent of the ring of integers considered, which suggested that analogous 
phenomena to the one observed by Dung and Voll in~\cite[Theorem~A]{DuVo14} also hold for bivariate representation and bivariate conjugacy class zeta functions.

\subsection{Structure of the paper}
In the preliminary Section~\ref{convgen}, we recall basic concepts and results about functions on two comples variables, double series and convergence and meromorphy of nonuniform 
Euler products on two variables. 

Section~\ref{spadic} is divided in two parts: in section~\ref{padic}, we recall from~\cite{PL18} how to write most local factors of the bivariate 
zeta functions of Definition~\ref{biva} in terms of $p$-adic integrals and, in Section~\ref{Denef}, we show the above mentioned $p$-adic integrals may be described in terms of 
formulae of Denef type~\eqref{deneftype}, which will be used in Section~\ref{pthmA} to determine the domains of convergence and meromorphy of the 
bivariate zeta functions. 

In Section~\ref{pthmA}, we prove the main theorem, giving the proof of Theorem~\ref{thmA}\eqref{parti} in Section~\ref{holomorphic} and of Theorem~\ref{thmA}\eqref{partii} 
in Section~\ref{meromorphic}.

\section{Background}\label{convgen} 
\subsection{Two complex variables}\label{holmer}
In this section, we recall briefly the meaning of holomorphy and meromorphy for complex functions on two variables. We refer the reader to~\cite{FiLi12, Gauthier} for further 
information about functions on several complex variables.
We call \emph{domain} a connected open subset of~$\C^2$ (with the usual topology). 
\begin{dfn}
Let $U\subseteq \C^2$ be an open set. A continuous function $f:U \to \C$ is \emph{holomorphic} if it is holomorphic in each variable. Equivalently, $f$ is holomorphic 
if it satisfies the system of homogeneous equations $\tfrac{\partial f}{\partial \overline{z_j}}=0$, for $j=1,2$, where for $\re(z_j)=x_j$ and $\im(z_j)=y_j$, 
\[\frac{\partial }{\partial \overline{z_j}}=\frac{1}{2}\left(\frac{\partial}{\partial x_j}+i\frac{\partial}{\partial y_j}\right).\]
\end{dfn}
\begin{exm}
The function $f:\C^2 \to \C$ given by $f(z_1,z_2)=az_1+bz_2+c$, where $a$ and $b$ are nonzero real numbers and $c \in \R$, is holomorphic on the whole~$\C^2$. 
Its zero set is 
\[V(f)=\{(z_1,z_2) \in \C^2 \mid az_1+bz_2=-c\}.\]
In particular, the function $g = \tfrac{1}{f}$ has set of poles $V(f)$. 
\end{exm}

In the one variable case, a function is meromorphic on a certain domain if it is locally the quotient of two holomorphic functions such that the denominator is nonzero. 
In particular, a meromorphic function may only have finite-order isolated poles.  
In the example above, we see that the rational function~$g$ has infinitely many poles and none of them is isolated. However, we shall see that $g$ is a meromorphic function on
the whole~$\C^2$. This is because meromorphy on several complex variables allows for set of (non-isolated) poles, as long as this set is sufficiently ``small''. 
More precisely, we call a subset $M$ of a domain $\Omega \subset \C^2$ \emph{thin} if it is relatively closed on~$\Omega$, that is, an intersection of a closed subset with any set, 
and if for each $\mathbf{z}=(z_1,z_2)\in \C^2$ there is a neighbourhood $U_{\mathbf{z}}$ of $\mathbf{z}$ and a holomorphic function $f_{\mathbf{z}}$ such that 
$M \cap U_{\mathbf{z}} \subset V(f_{\mathbf{z}})$. 
Particularly, if $f : \Omega \to \C$ is holomorphic, then $V(f):=\{\mathbf{z}\in \C^2 \mid f(\mathbf{z})=0\}$ is a thin set. 

\begin{dfn}\label{meromorphy}\cite[Definition~2.1 of Chap.~VI]{FiLi12} A \emph{meromorphic function} on a domain~$\Omega \subset \C^2$ is a function 
$f: \Omega \to \C$ such that there exists a thin set~$M \subset \Omega$ for which $f$ is holomorphic on $\Omega\setminus M$ and, for each $\mathbf{z}_0\in \Omega$, 
there exist a neighbourhood $U_{\mathbf{z}_0}$ of $\mathbf{z}_0$ in $\Omega$ and holomorphic functions $g,h : U_{\mathbf{z}_0} \to \C$ with 
$g\not\equiv 0$ such that $V(h)\subset M$ and \[f(\mathbf{z})=\frac{g(\mathbf{z})}{h(\mathbf{z})}, \text{  for }\mathbf{z}\in U\setminus M.\]
\end{dfn}

In particular, we see that if $f(\mathbf{z})=\tfrac{g(\mathbf{z})}{h(\mathbf{z})}$ with $g,h: \Omega \to \C$ holomorphic and $h \not\equiv 0$, then, since $V(h)$ is thin, 
$f$ is meromorphic on~$\Omega$.

The following result states that the complement of a thin set in a domain is also a domain. 
\begin{pps}\cite[Proposition~1.3 of Chap.~VI]{FiLi12}
Let~$M$ be a thin subset of a domain $\Omega \subseteq \C^2$. Then~$\Omega \setminus M$ is connected.
\end{pps}

\subsection{Double series}\label{dseries}
In this section, we recall some properties of double series which are needed. We refer the reader to~\cite[Section~7]{GhLi} for further results and definitions on double sequences 
and  double series. For simplicity we write $(a_{m,n})=(a_{m,n})_{m,n\in\N}$.

We observe that a (single) series $(a_n)_{n\in \N}$ can be regarded as a double series $(a_{m,n})$ by defining $a_{1,n}=a_n$, and $a_{m,n}=0$, for all $n \in \N$ and 
$m\in \N_{>1}$.
In particular, the results on double series also hold for (single) series.
The converse does not hold. For instance, in contrast with single sequences, a convergent double sequence need not be bounded.
An example in which this property fails is the double sequence of terms $a_{n,1}=n$ and $a_{n,m}=1$, for $n\in \N$ and $m \in \N_{>1}$. 

However, a double series $\sum\sum_{(m,n)} a_{m,n}$ with nonnegative coefficients is convergent 
if and only if the double sequence $(A_{n,m})$ of its partial sums $A_{m,n}:=\sum_{k=1}^{m}\sum_{l=1}^{n}a_{k,l}$ is bounded above; \cite[Proposition~7.14]{GhLi}.
Also, a monotonic double sequence is convergent if and only if it is bounded; see~\cite[Proposition~7.4]{GhLi}.

For the sake of completeness, we show the following Lemmata, which are analogous to similar results on single series.
\begin{lem}\label{product} Let $(a_{m,n})$ be a bounded double sequence and let $\sum\sum_{(m,n)}b_{m,n}$ be an absolutely convergent double series. 
Then $\sum\sum_{(m,n)}a_{m,n}b_{m,n}$ converges absolutely.
\end{lem}
\begin{proof}
 There exists $M>0$ such that $|a_{m,n}|<M$ for all $m,n\in\N$. Since the monotonically non-decreasing double sequence $(\sum_{k=1}^{m}\sum_{l=1}^{n}|b_{k,l}|)_{m,n}$ converges, 
 it is bounded by a positive real number $N$. Therefore, 
 \[\sum_{k=1}^{m}\sum_{l=1}^{n}|a_{k,l}b_{k,l}|<M\sum_{k=1}^{m}\sum_{l=1}^{n}|b_{k,l}|<MN. \qedhere\]
\end{proof}

\begin{lem}\label{sumprod} A double series $\sum\sum_{(m,n)}a_{m,n}$ converges absolutely if and only if the product $\prod\prod_{(m,n)}1+a_{m,n}$ converges 
absolutely. 
\end{lem}
\begin{proof}
Denote by $P_{m,n}$ the partial product $\prod_{k=1}^{m}\prod_{l=1}^{n}(1+|a_{k,l}|)$ and by $S_{m,n}$ the partial sum $\sum_{k=1}^{m}\sum_{l=1}^{n}|a_{k,l}|$.
The double sequences $(P_{m,n})$ and $(S_{m,n})$ are positive non-decreasing double sequences and hence they converge if and only if they are bounded.
One the one hand, since $1+x \leq e^x$ for all $x \in \R_{\geq 0}$, it follows that
\[P_{m,n}=\prod_{k=1}^{m}\prod_{l=1}^{n}(1+|a_{k,l}|) < \prod_{k=1}^{m}\prod_{l=1}^{n}e^{|a_{k,l}|}=e^{s_{m,n}}.\]
On the other hand, it is easy to see that $P_{m,n} \geq 1+S_{m,n}$. Therefore, $(P_{m,n})$~is bounded if and only if~$(S_{m,n})$ is bounded.
\end{proof}
\subsection{Convergence of bivariate Euler products}\label{deuler}
In this section we recall from~\cite[Theorem~2.7]{Del14} the domains of convergence and meromorphy of the Euler products on several variables.
In that article, Delabarre deals with nonuniform Euler products on~$n>1$ variables $s_1$, $\dots$, $s_n$. 
We observe that Delabarre's main results admit straightforward generalisations to products over prime ideals of~$\ri$, but we illustrate this just for the case~$n=2$.

For each nonzero prime ideal~$\p$ of~$\ri$, denote by~$q_{\p}$ the cardinality of the residue field~$\ri/\p$. 
We are interested in the domains of convergence and meromorphy of the Euler products 
\[Z_{c}(s_1,s_2)=\prod_{\p \in P}h(q_{\p}^{-s_1}, q_{\p}^{-s_2}, q_{\p}^{-c}),\]
where~$c$ is a fixed nonzero integer and $h(X_1,X_2,X_3)\in \Z[X_1,X_2,X_3]$ is a polynomial 
\[h(X_1,X_2,X_3)=1+\sum_{j=1}^{r}a_jX_{1}^{\alpha_{1,j}}X_{2}^{\alpha_{2,j}}X_{3}^{\alpha_{3,j}},\] 
with $a_j\neq 0$ and $\widehat{\boldsymbol{\alpha}}_j=(\alpha_{1,j},\alpha_{2,j},\alpha_{2,j})\in \Z^3\setminus\{\mathbf{0}\}$, for each $j \in [r]$, 
where for each $n\in\N$, $[n]$ denotes the set $\{1,\dots, n\}$. 
Denote $\boldsymbol{\alpha}_j=(\alpha_{1,j},\alpha_{2,j})$.

The polynomial $h(X_1,X_2,X_3)$ is called \emph{cyclotomic} if there exists a finite set $I \subset \N^{n+1} \setminus \{0\}$ such that
\[h(X_1,X_2,X_3)= \prod_{\boldsymbol{\lambda}=(\lambda_1, \lambda_2, \lambda_3) \in I}(1-X_{1}^{\lambda_1}X_{2}^{\lambda_2}X_{3}^{\lambda_3})^{\gamma(\boldsymbol{\lambda})},\]
where the $\gamma(\boldsymbol{\lambda})$ are nonzero positive integers.
If $h$ is cyclotomic, then $Z_{c}(s_1,s_2)$ can be meromorphically continued to the whole $\C^2$. For this reason, from now on, we assume that $h$ is not constant and does not 
contain cyclotomic factors.

For each $\delta\geq 0$, set
\begin{align*}
 W_c(\delta)&=\{(s_1,s_2)\in \C^2 \mid \re(\alpha_{1,j}s_1+\alpha_{2,j}s_2)>\delta-c\alpha_{j,3},~j\in[r]\}.
\end{align*}
\begin{pps}\label{Del}\cite[Theorem~2.7]{Del14} The product $(s_1,s_2) \mapsto Z_c(s_1,s_2)$ converges absolutely in the domain $W_c(1)$ and admits meromorphic continuation 
to $W_c(0)$. 
\end{pps}

Set $\hat{h}(X_1,X_2,X_3)=\sum_{j=1}^{r}a_jX_{1}^{\alpha_{1,j}}X_{2}^{\alpha_{2,j}}X_{3}^{\alpha_{3,j}}=h(X_1,X_2,X_3)-1$.
Lemma~\ref{sumprod} then yields that the sum 
 \[S_{c}(s_1,s_2)=\sum_{\p \in P}\hat{h}(q_{\p}^{-s_1},q^{-s_2},q^{-c})\]
converges absolutely in the domain $W_c(1)$. 

\begin{rmk}
 Let $P \subseteq \specp$ be a finite set of prime ideals of~$\ri$. Since 
\[\prod_{\p \in P}h(q_{\p}^{-s_1}, q_{\p}^{-s_2}, q_{\p}^{-c})\] 
is analytic, the infinite product
\[p(s_1,s_2):=\prod_{\p \notin P}h(q_{\p}^{-s_1}, q_{\p}^{-s_2}, q_{\p}^{-c})=\frac{Z_{c}(s_1,s_2)}{\prod_{\p \in P}h(q_{\p}^{-s_1}, q_{\p}^{-s_2}, q_{\p}^{-c})}\] 
also admits meromorphic continuation to $W_{c}(0)$. It converges on~$W_{c}(1)$ if the set of zeros~$V(p)$ of~$p(s_1,s_2)$ is not contained in this domain.
It follows that Proposition~\ref{Del} holds if we consider $Z_{c}(s_1,s_2)$ as a product over almost all nonzero prime ideals of~$\ri$, 
as long as the zeros of the corresponding $h(q_{\p}^{-s_1}, q_{\p}^{-s_2}, q_{\p}^{-c})$ do not lie in~$W_{c}(1)$. 
\end{rmk}
\section{$\p$-adic integrals}\label{spadic}
Almost all local factors of bivariate representation and conjugacy class zeta functions can be written in terms of $\p$-adic integrals 
which are special cases of the general integrals defined in \cite[Section~2]{Vo10}. 
In~\cite[Theorem~2.2]{Vo10} and~\cite[Proposition~3.3]{AKOV13}, for almost all nontrivial prime ideals $\p$, these general $\p$-adic integrals are described in terms of 
formulae of Denef type, from which the domains of convergence can be read off. By a formula of Denef type we mean a finite sum of the form 
\begin{equation}\label{deneftype}\sum_{i=1}^{n}|\overline{V_i}(\lrip)|W_i(q,q^{s_1},q^{s_2}),\end{equation} 
where  $|V_i(\lrip)|$ denotes the $\op$-rational points of reductions modulo $\p$ of a suitable algebraic variety $V_i$ defined over $\ri$ and $W_i(X,Y,Z)$ is a rational function.

In Section~\ref{pthmA}, we use these descriptions to prove Theorem~\ref{thmA}.
In preparation for this, we recall from~\cite[Section~4.2]{PL18} how to write the local factors of the bivariate zeta functions of Definition~\ref{biva} in 
terms of $\p$-adic integrals in Section~\ref{padic}, and recall the general integrals of~\cite[Section~2]{Vo10}, as well as~\cite[Theorem~2.2]{Vo10} 
and~\cite[Proposition~3.3]{AKOV13}, in Section~\ref{Denef}.

\subsection{Poincaré series and $p$-adic integration}\label{padic}
Fix a prime ideal $\p$ of $\ri$ and write $\lri=\ric$. Set $\g=\Lambda(\lri)=(\Lambda \otimes_{\ri}\ri)\otimes_{\ri}\lri$. Write $\z$ for the centre of $\g$ and denote by $\g'$ 
its derived Lie sublattice.
Set also
\[h=\rk_{\lri}(\g), \hspace{0.5cm} a=\rk_{\lri}(z), \hspace{0.5cm} b=\rk_{\lri}(\g').\]
Fix an ordered set $\mathbf{e}=(e_1,\dots, e_a)$ of elements of~$\g$ such that $\overline{\textbf{e}}=(\overline{e}_1, \dots, \overline{e}_a)$ is an ordered set of $\lri$-module 
generators of~$\g/\z$, where $\overline{\phantom{a}}$ denotes the canonical map $\g \to \g/\z$. Let also $\mathbf{f}=(f_1, \dots, f_b)$ be an ordered set of $\lri$-module generators 
of~$\g'$. 
Recall the notation $[n]=\{1,\dots,n\}$ for $n\in\N$. For $i,j \in [a]$ and $l\in[b]$, consider the structure constants~$\lambda_{ij}^{l}$ with respect to 
the bases~$\mathbf{e}$ and~$\mathbf{f}$, given by 
\[[e_i,e_j]=\sum_{l=1}^{b}\lambda_{ij}^{l}f_l.\]
\begin{dfn}\label{commutator}\cite[Definition~2.1]{ObVo15} The $A$-commutator matrix and the $B$-commutator matrix of $\g$ with respect to $\mathbf{e}$ and $\mathbf{f}$ are, respectively,
\begin{align*}
 A(X_1,\dots,X_a)	&=\left( \sum_{k=1}^{a}\lambda_{ik}^{j}X_k\right)_{ij}\in \Mat_{a \times b}(\lri[\underline{X}]),\text{ and} \\
 B(Y_1, \dots, Y_b)	&=\left( \sum_{k=1}^{b}\lambda_{ij}^{k}Y_k\right)_{ij}\in \Mat_{a \times a}(\lri[\underline{Y}]),
\end{align*}
where $\underline{X}=(X_1, \dots, X_a)$ and $\underline{Y}=(Y_1, \dots, Y_b)$ are independent variables. 
\end{dfn}

Suppose that $(p,c)\neq (2,3)$. In \cite[Section~4.2]{PL18}, the bivariate conjugacy class zeta functions of groups of the form $\G(\ri)$ are described 
in terms of Poincaré series encoding the elementary divisor type of the $A$-commutator matrix for some bases $\mathbf{e}$ and $\mathbf{f}$ as above, whilst the 
bivariate representation zeta function of $\G(\ri)$ is described in terms of Poincaré series encoding the elementary divisor type of the $B$-commutator matrix.
As a consequence, the local factors of the bivariate zeta functions can be written in terms of $p$-adic integrals, which we now recall.
For a given matrix $\mathcal{R}(\underline{Y})=\mathcal{R}(Y_1, \dots, Y_n)$ of polynomials $\mathcal{R}(\underline{Y})_{ij}\in\lri[\underline{Y}]$, define
\begin{equation}\label{Vo10}\pint_{\lri,\mathcal{R}}(\underline{r},t)=\frac{1}{1-q^{-1}}\int_{(x,\underline{y}) \in \p\times \W_{n}}|x|_{\p}^{t-n-1}\prod_{k=1}^{u_{\mathcal{R}}}
\frac{\|F_{k}(\mathcal{R}(\underline{y})) \cup xF_{k-1}(\mathcal{R}(\underline{y}))\|_{\p}^{r_k}}{\|F_{k-1}(\mathcal{R}(\underline{y}))\|_{\p}^{r_{k}}}d\mu,
\end{equation}
where $\underline{r}=(r_1, \dots, r_{\epsilon})$ is a vector of variables, $\W_{n}:=\lri^n\setminus \p^n$, and $\mu$ is the additive Haar measure on $\lri^{n+1}$, 
normalised so that $\mu(\lri^{n+1})=1$. 
Moreover, $u_{\mathcal{R}}=\max\{rk_{\text{Frac}(\lri)}\mathcal{R}(\textbf{z}) \mid \textbf{z}\in \lri^n\}$, and
$F_j(\mathcal{R}(\underline{y}))$ denotes the set of $(j \times j)$-minors of $\mathcal{R}(\underline{y})$.

\begin{pps}\label{intpadic}\cite[Proposition~4.8]{PL18} For $(p,c)\neq (2,3)$,
\begin{align}
\label{rpadic}\brzf_{\G(\lri)}(s_1,s_2)=
&\frac{1}{1-q^{r-s_2}}\left(1+ \pint_{\lri,B}\left(\tfrac{-s_1-2}{2},\ub s_1+s_2+2\ub-h-1\right)\right),\\ 
\label{ccpadic}\bcczf_{\G(\lri)}(s_1,s_2)=
&\frac{1}{1-q^{z-s_2}}\left(1+ \pint_{\lri,A}(-s_1-1,\ua s_1+s_2+\ua-h-1)\right).
\end{align}
where $h=\rk_{\lri}(\g)$, $r=\rk_{\lri}(\g/\g')=h-b$, $z=\rk_{\lri}(\z)=h-a$, and 
\begin{align*}\ua&=\max\{\rk_{\text{Frac}(\lri)}A(\mathbf{z}) \mid \mathbf{z}\in \lri^a\},\\ 
 2\ub&=\max\{\rk_{\text{Frac}(\lri)}B(\mathbf{z}) \mid \mathbf{z}\in \lri^b\}.
\end{align*}
\end{pps}
%

When determining the domains of convergence and meromorphic continuation of the bivariate zeta functions, it suffices to determine the respective domains of convergence of 
the following functions:
\begin{align}
&\label{rpadic2}\rrel=1+\pint_{\lri,B}\left(\tfrac{-s_1-2}{2},\ub s_1+s_2+2\ub-h-1\right),\\ 
&\label{ccpadic2}\crel=1+ \pint_{\lri,A}(-1-s_1,\ua s_1+s_2+\ua-h-1).
\end{align}
In fact, the products $\prodq(1-q^{r-s_1})^{-1}$ and $\prodq(1-q^{z-s_1})^{-1}$ converge, respectively, for $\re(s_2)>1-r$ and $\re(s_2)>1-z$, and both admit meromorphic continuation to 
the whole~$\C^2$.
We call the functions~\eqref{rpadic2} and~\eqref{ccpadic2} the \emph{main terms} of the bivariate representation, respectively, of the bivariate conjugacy class zeta functions 
of~$\G(\lri)$.
\subsection{Formulae of Denef type}\label{Denef}
Firstly, we recall some notation from~\cite[Section~2]{Vo10}. 
Fix $d,l \in \N$ and set 
\begin{enumerate}
 \item $J_{\kappa}$ a finite index set, for each $\kappa \in [l]$,
 \item $e_{\kappa\iota}\in \Z_{\geq0}$, for each $\kappa \in [l]$ and $\iota \in J_{\kappa}$,
 \item $F_{\kappa\iota}(\underline{Y})=F_{\kappa\iota}(Y_1, \dots, Y_m)\in \lri[\underline{Y}]$, for all $\kappa \in [l]$ and $\iota \in J_k$.
\end{enumerate}
Let also $\mathcal{W}(\lri)\subseteq \lri^d$ be a union of cosets modulo $\p^{(d)}$. Define
\begin{equation}\label{Vogen}
 Z(\underline{s})=\int_{\p\times \Gl_d(\lri)} \prod_{\kappa=1}^{l}\left|\left|\bigcup_{\iota \in J_{\kappa}} 
 X^{e_{\kappa\iota}}F_{\kappa\iota}(\underline{Y})\right|\right|_{\p}^{s_{\kappa}}d\mu,
\end{equation}
where $\underline{s}=(s_1, \dots, s_l)$ is a vector of complex variables and $X$ and $\underline{Y}=(Y_1, \dots, Y_d)$ are independent integration variables. 
The numbers $d$, $l$, as well as the data $J_{\kappa}$, $e_{\kappa\iota}$, and $F_{\kappa\iota}(\underline{Y})=F_{\kappa\iota}(Y_1, \dots, Y_d)\in \lri[\underline{Y}]$ 
will be refered to as \emph{the data associated to the integral} $Z(\underline{s})$. 
The term $d\mu$ denotes the additive Haar measure on $\lri^{d+1}$, normalised so that 
\[d\mu(\p\times\Gl_d(\lri))=q^{-1}\prod_{\theta=1}^{d-1}(1-q^{-\theta}).\]

Assume further that the ideals $(F_{\kappa\iota})$ are invariant under the natural action of the standard Borel subgroup $B \subseteq \Gl_d$.

Although the integral (\ref{Vogen}) is a $\p$-adic integral, that is, a ``local'' object, its integrand is defined globally, in contrast with the integrals of 
expressions~\eqref{rpadic} and~\eqref{ccpadic}, whose integrands are defined in terms of the commutator matrices $B(\underline{Y})$ and $A(\underline{X})$, respectively.
These matrices are defined with respect to the bases~$\mathbf{e}$ and~$\mathbf{f}$, which are defined only locally. 
A global source of local bases~$\mathbf{e}$ and~$\mathbf{f}$---such as the ones of \cite[Section~3.3]{DuVo14}---is: we choose an $\ri$-basis~$\mathbf{f}$ for a free 
$\ri$-submodule of finite index of the isolator~$i(\Lambda')$ of the derived $\ri$-Lie sublattice of~$\Lambda$. 
By \cite[Lemma~2.5]{StVo14}, $\mathbf{f}$~can be extended to an $\ri$-basis~$\mathbf{e}$ for a free finite-index $\ri$-submodule of~$\Lambda$, which we denote by~$M$. 
If the residue characteristic~$p$ of~$\p$ does not divide $|\Lambda:M|$ or $|i(\Lambda'): \Lambda'|$, this basis~$\mathbf{e}$ may be 
used to obtain an $\lri$-basis for~$\Lambda(\lri)$, by tensoring the elements of~$\mathbf{e}$ with~$\lri$. 

In \cite[Section~4.4]{PL18}, it is shown that, for each $\ast \in \{\irrup,\ccup\}$, there are vectors $\asta_1$, $\asta_2$, $\astb$, and certain choices of data associated to 
$Z(\underline{s})$ such that
\begin{align}\rrel
\label{rpintz}&=\frac{1}{\prod_{\theta=1}^{d-1}(1-q^{-\theta})}Z(\mathbf{a}_{1}^{\irrup}s_1+\mathbf{a}_{2}^{\irrup}s_2+\mathbf{b}^{\irrup}),\\
\crel
\label{ccpintz}&=\frac{1}{\prod_{\theta=1}^{d-1}(1-q^{-\theta})}Z(\mathbf{a}_{1}^{\ccup}s_1+\mathbf{a}_{2}^{\ccup}s_2+\mathbf{b}^{\ccup}).
\end{align}
Consequently, given a nonzero prime ideal~$\p$ with residue field of characteristic $p \nmid |\Lambda: M||i(\Lambda'):\Lambda'|$,
for any extension $\Lri/\lri$ with degree of inertia $f=f(\Lri/\lri)$, the main term $\arel$ is given by 
\begin{equation}\label{mainterms}
 \widetilde{\bzfs_{\G(\Lri)}}(s_1,s_2)=\frac{1}{\prod_{\theta=1}^{d-1}(1-q^{-f\theta})}Z(\asta_{1}s_1+\asta_{2}s_2+\astb).
\end{equation}
 
We now want to write the main terms of the bivariate representation and the bivariate conjugacy class zeta functions in terms of formulae of Denef type such as~\eqref{deneftype}.
To state these results, we introduce some notation. 

Consider the $\ri$-ideal
\[\mathcal{I}=\prod_{\kappa=1}^{l}\prod_{\iota \in J_{\kappa}}(F_{\kappa\iota}),\] where $F_{\kappa\iota}(\underline{Y})=F_{\kappa\iota}(Y_1, \dots, Y_d)\in \lri[\underline{Y}]$ 
are the functions appearing in the integrand of~\eqref{Vogen}. 
Fix a principalization~$(Y,h)$ of~$\mathcal{I}$ such that $h: Y \to \Gl_d/B$.

Let $T$ denote a finite set indexing the irreducible components $E_u$ of the pre-image of $h$ of the subvariety of $\Gl_d/B$ defined by $\mathcal{I}$. 
The numerical data associated to $(Y,h)$ is $(N_{u\kappa\iota},\nu_u)_{u\kappa\iota}$, where $N_{u\kappa\iota}$ denotes the multiplicity of the irreducible component $E_u$ in 
the pre-image under~$h$ of the variety defined by the ideal~$(F_{\kappa\iota})$ and~$\nu_u-1$ denotes the multiplicity of~$E_u$ in the divisor $h^{\ast}(d\mu(\mathbf{y}))$.

When rewriting the main terms of the bivariate zeta functions in terms of Denef formulae~\eqref{deneftype}, the rational functions $W_i(X,Y,Z)$ are the functions 
$\Xi_{U,(d_{\kappa\iota})}^{N}$ defined below in terms of the numerical data $(N_{u\kappa\iota},\nu_u)_{u\kappa\iota}$ associated to~$(Y,h)$.

\begin{dfn}\label{Xi}
Let $N\in \N$, $U\subseteq T$, $(d_{\kappa\iota})\in \N_{0}^{\prod_{\kappa=1}^{l}J_{\kappa}}$, and write $\m=((m_u)_{u \in U},m_{t+1}) \in \N^{|U|}\times\N$. Define 
\[\Xi_{U,(d_{\kappa\iota})}^{N}(q,\underline{s})=(1-q^{-1})^{d+1}q^{-N{d \choose 2}}\sum_{\m \in \N_{\geq N}^{|U|}\times\N}q^{\lfct(\m)-
\sum_{\kappa=1}^{l}s_{\kappa}\min\{\lfct_{\kappa\iota}(\m)-d_{\kappa\iota}\mid \iota \in J_{\kappa}\}},\]
where 
\begin{align*}
\lfct(\mathbf{m})&= m_{t+1}+\sum_{u \in U}\nu_um_u,\\
 \lfct_{\kappa\iota}(\mathbf{m})&= e_{\kappa\iota}m_{t+1}+\sum_{u \in U}N_{u\kappa\iota}m_u, \text{ for }\kappa \in [l], \iota\in J_{\kappa}.
\end{align*}
For the special case $N=1$ and $(d_{\kappa\iota})=(0)$, denote $\Xi_{U}(q,\underline{s}):=\Xi_{U,(0)}^{1}(q,\underline{s})$. 
\end{dfn}

\begin{pps}\label{ppsDenef}
\begin{enumerate}
  \item (1)~\cite[Theorem~2.2]{Vo10} If~$(Y,h)$ has good reduction modulo~$\p$, then 
\[Z(\underline{s})=\frac{(1-q^{-1})^{d+1}}{q^{d \choose 2}}\sum_{U\subseteq T}c_U(\lri/\p)(q-1)^{|U|}\Xi_U(q,\underline{s}).\]\vspace{0.2cm}
  \item (2)~\cite[Proposition~3.3]{DuVo14} If~$(Y,h)$ has bad reduction modulo~$\p$, there exist $N \in \N$, finite sets $J \subset \N_0$, and 
  $\Delta \subset \N_{0}^{\prod_{\kappa=1}^{l}J_{\kappa}}$ such that 
\[
Z(\underline{s})= \frac{(1-q^{-1})^{d+1}}{q^{N{d \choose 2}}}\sum_{\substack{U \subseteq T, j \in J\\ (d_{\kappa\iota})\in\Delta}}
c_{U,j,(d_{\kappa\iota})}(\lri/\p)(q^N-q^{N-1})^{|U|}q^{-j}\Xi_{U,(d_{\kappa\iota})}^{N}(q,\underline{s}).
\]
Here, $c_{U}(\op)$ and $c_{U,j,(d_{\kappa\iota})}(\lri/\p)$ denote the numbers of $\lri/\p$-rational points of certain varieties over $\lri/\p$.
\end{enumerate}
\end{pps}
%

\section{Convergence and meromorphic continuation}\label{pthmA}
Throughout Section~\ref{pthmA}, we use the notation introduced in Section~\ref{Denef}. 
\subsection{Convergence -- proof of Theorem~\ref{thmA}\eqref{parti}}\label{holomorphic}
The principalization $(Y,h)$ 
has good reduction modulo $\p$ for all but finitely many nonzero prime ideals~$\p$ of~$\ri$. 
We denote by~$Q_1$ the set of all nonzero prime ideals~$\p$ such that $(Y,h)$ has bad reduction modulo~$\p$, and by~$Q_2$ the set of all nonzero prime ideals~$\p$ 
of~$\ri$ with residue field of characteristic~$p$ satisfying:\vspace{-0.1cm}
\begin{enumerate}
 \item\label{q2} $p$ divides $|\Lambda:M||\iota(\Lambda'):\Lambda'|$,\text{ or }
 \item\label{q3} $(p,c)=(2,3)$.
\end{enumerate}
Denote by~$Q$ the finite set $Q_1\cup Q_2$ of `bad primes'.
We divide the proof of Theorem~\ref{thmA}\eqref{parti} into the cases $\p \notin Q$ and $\p \in Q_1$. 
The primes belonging to $Q_2$ are the primes excluded in this theorem.
\subsubsection{Good reduction}\label{goodreduction}
We now determine the maximal domain of convergence $\hol$ of
\[\astgc=\prodq \arel.\]

In \cite[Section~3.1]{DuVo14}, the authors rewrite the functions $\Xi_U(q,\underline{s})$---given in~\eqref{Xi}---in terms of zeta functions of polyhedral cones in a fan, 
allowing them to deduce a formula for the integral $Z(\underline{s})$, from which one can read off the domain of convergence. 
In the following, we apply this formula to the integral $Z(\asta_{1}s_1+\asta_{2}s_2+\astb)$ appearing in~\eqref{mainterms} to determine the domain of convergence of $\astgc$. 
We recall from \cite[Section~3.1]{DuVo14} the notation needed.

Let~$t$ be the cardinality of the set~$T$ defined in Section~\ref{Denef}. Let $\{R_i\}_{i \in [w]_0}$ be a finite triangulation of $\R_{\geq 0}^{t+1}$ consisting of 
pairwise disjoint cones~$R_i$ such that each of them is a relatively open simple rational polyhedral cone with the property of eliminating the ``min-terms'' in the exponent 
of~$q$ in $\Xi_U(q,\underline{s})$. Assume that $R_0=\{0\}$ and that $R_1, \dots, R_{z}$ are the one-dimensional cones (also called \emph{rays}) in this triangulation. 
Denote by $\mathbf{r}_j \in \N_{0}^{t+1}$ the shortest integral vector on the cone~$R_j$, for each $j \in [z]$. Then $R_j=\R_{>0}\mathbf{r}_j$.

All cones $R_i$ are generated by rays, so that for each $i \in [w]$ there exists a set $M_i\subseteq [z]$ such that $R_i$ is the direct sum of monoids
\[R_i=\bigoplus_{j \in M_i}\R_{>0}\mathbf{r}_j.\]
Since $R_i=\R_{>0}\mathbf{r}_i$ exactly when $i \in [z]$, it follows that $|M_i|=1$ if and only if $i \in [z]$. Because the $R_j$ are simple, 
\[R_i\cap\N_{0}^{t+1}=\bigoplus_{j \in M_i}\N\mathbf{r}_j.\]

For $U \subseteq T$, the domain of summation of $\Xi_U(q,\underline{s})$ is
\[\mathscr{C}_U=\{\underline{m}\in \N_{0}^{t}\times\N \mid m_u =0 \text{ if and only if }u\in T\backslash U\}.\]
Denote by $W'_{U}$ the (unique) subset of $[w]$ such that 
\[\mathscr{C}_{U}=\bigcupdot_{i \in W'_{U}}R_i\cap\N_{0}^{t+1},\] 
and by $W'$ the union of all $W'_{U}$, that is, $W' \subseteq [w]$ is the set of index of cones which do not lie in the boundary component $\R_{\geq 0}\times \{0\}$ of 
$\R_{\geq 0}^{t+1}$.

Given $i\in W'$, denote by $U_i$ the unique subset $U \subseteq T$ such that $i \in W'_{U}$, and $c_i(\lri/\p):=c_{U_i}(\lri/\p)$. 
\begin{pps}\label{goodred}\cite[Proposition~3.2]{DuVo14} For $\p \notin Q$, there exist $\mathscr{A}_{j\kappa}\in \N_0$ and $\mathscr{B}_{j} \in \N$ for each 
$j \in [z]$ and $\kappa \in [l]$
such that 
\[Z(\underline{s})=\frac{(1-q^{-1})^{d+1}}{q^{d \choose 2}}\sum_{i \in W'}c_i(\lri/\p)(q-1)^{|U_i|}\prod_{j \in M_i}\frac{q^{-(\sum_{\kappa=1}^{l}
\mathscr{A}_{j\kappa}s_k+\mathscr{B}_j)}}{1-q^{-(\sum_{\kappa=1}^{l}\mathscr{A}_{j\kappa}s_\kappa+\mathscr{B}_j)}}.\]	
\end{pps}

Proposition~\ref{goodred} applied to~\eqref{mainterms} gives the following result.
\begin{pps}\label{goodp} For $\p\notin Q$, there exist $A_{1j}^{\ast}$, $A_{2j}^{\ast}$, $B_{j}^{\ast}\in\Q$, for each $j \in [z]$, such that 
$\arel$ is given by
\[
1+\frac{(1-q^{-1})^dq^{-{d \choose 2}}}{\prod_{r=\theta}^{d-1}(1-q^{-\theta})}\sum_{i \in W'}c_i(\lri/\p)(q-1)^{|U_i|}
\prod_{j \in M_i}\frac{q^{-A_{1j}^{\ast}s_1-A_{2j}^{\ast}s_2-B_{j}^{\ast}}}{1-q^{-A_{1j}^{\ast}s_1-A_{2j}^{\ast}s_2-B_{j}^{\ast}}}.
\] 
\end{pps}
\begin{rmk}\label{AjkBj}
The numbers $\mathscr{A}_{j\kappa}$ of Proposition~\ref{goodred} are constructed so that $\sum_{j \in M_i}\mathscr{A}_{j\kappa}=0$ if and only if the cone $R_i$ lies in the boundary 
component $\R_{\geq 0}\times \{0\}$, that is, if and only if $i \notin W'$. 
Moreover, all the $\mathscr{B}_j$ are nonnegative; see their construction in~\cite[Section~3.1]{DuVo14} and \cite[Remark~3.6]{DuVo14}. 
Similar arguments show that  $\astA_{1j}$, $\astA_{2j}$ of Proposition~\ref{goodp} are such that $\sum_{j \in M_i}A_{1j}$ and 
$\sum_{j \in M_i}A_{2j}$ are zero if and only if $i \notin W'$ and, moreover, all $\astB_j$'s of Proposition~\ref{goodp} are nonnegative. 
\end{rmk}

The numbers $c_i(\lri/\p)$ are all divisible by \[q^{d-1 \choose 2}(1-q^{-1})^{d-1}\prod_{\theta=1}^{d-1}(1-q^{-\theta}),\]
because of the way of construction of the relevant integrals; see~\cite[Remark~3.4]{DuVo14}.

Proposition~\ref{goodp} shows that the poles of the main terms of the bivariate zeta functions are the ones occurring in the terms 
\[(1-q^{\astA_{1j}s_1-\astA_{2j}s_2-\astB})^{-1},\] for $j \in M_i$ and $i \in W'$ such that $(\astA_{1j},\astA_{2j})\neq (0,0)$. 
Since $(\astA_{1j},\astA_{2j}) \neq (0,0)$ exactly when $j \in W'$, it follows that the poles of $\arel$ are unions of sets 
\[\Ps_j=\{(s_1,s_2) \mid \astA_{1j}s_1+\astA_{2j}s_2=\astB_j\},~j \in [z]\cap W'.\]
%
%

Consequently, the domain of convergence of $\arel$ is a finite intersection of sets of the following form.
\begin{dfn}\label{Didelta}
For each $\delta \geq 0$ and each $i \in W'\cap[z]$, set 
\[\D_{i,\delta}=\{\setopt \mid \re(\dci)>1-\astB_i-\delta\}.\]
\end{dfn}

Proposition~\ref{goodp} shows that the generating function $\arel$ converges at least on the domain $\bigcap_{j \in [z]\cap W'}\D_{j,1}$.

For each $i \in W$, let
\begin{equation}\label{Zip}\Zip=\frac{(1-q^{-1})q^{-{d\choose 2}}}{\prod_{\theta=1}^{d-1}(1-q^{-\theta})}c_i(\lri/\p)(q-1)^{|U_i|}
\prod_{j \in M_i}\frac{q^{\powerj}}{1-q^{\powerj}}.\end{equation}
The $\Zip$ are ordinary generating functions in $q$, $q^{-s_1}$, and $q^{-s_2}$ with nonnegative coefficients. By Proposition~\ref{goodp},
\[\arel=1+\sumw\Zip.\]
Then,
\begin{equation}\label{goodZ}\astgc=\prodq\left(1+\sumw \Zip\right).\end{equation}
We now determine the domain of absolute convergence~$\D_{i}$ of $\prodq (1+\Zip)$, that is, of the sum $\sumq \Zip$. 
Since $W'$ is a finite set, $\astgc$~converges absolutely on $\bigcap_{i \in W'}\D_i$. 

In preparation for that, we need some notation. 
In the set-up of Section~\ref{Denef}, $T$~is the finite set of irreducible components~$E_u$ of the pre-image under~$h$ of the variety defined by~$\mathcal{I}$, and
$E_U:=\bigcap_{u \in U}E_u$. 
Denote by~$d_U$ the dimension of~$E_U$. For each $U \subseteq T$, it holds $d_U={d \choose 2}-|U|$; see~\cite[Proposition~4.13]{duSauGr}. 
The family of the irreducible components over~$K$ of~$E_U$ of maximal dimension~$d_U$ is denoted $\{F_{U,b}\}_{b \in I_U}$, where~$I_U$ is a finite set of indices. 
For $b \in I_U$, $l_{\p}(F_{U,b})$ denotes the number of irreducible components of~$\overline{F_{U,b}}$ over~$\lrip$ which are absolutely irreducible over an 
algebraic closure of~$\lrip$. 

We record a useful consequence of the Lang-Weil estimate given in~\cite[Proposition~8.9]{duSauGr}.
\begin{lem}\label{langweil}\cite[Proposition~8.9]{duSauGr} There exists $C>0$ such that for all $U \subseteq T$ and $\p \notin Q$,
\[\left|c_{U}(\op)-\sum_{b \in I_{U}}l_{\p}(F_{U,b})q^{d_{U}}\right|<Cq^{d_{U}}-\frac{1}{2},\] and $l_{\p}(F_{U,b})>0$ for a set of prime ideals with positive density.
This means that, for any sequence $(r_{\p})_{\p\notin Q}$ of rational numbers, a sum of the form $\sumq c_U(\op)r_{\p}$ converges absolutely if and only if 
$\sumq r_{\p}q^{-d_{U}}$ converges absolutely.
\end{lem}
In \cite[Proposition~4.9]{duSauGr}, it is shown that, for each $b \in I_U$, the number~$l_{\p}(F_{U,b})$ is positive for a set of prime ideals of positive density. 
We remark that the finitely many prime ideals excluded are elements of~$Q$; see the proof of~\cite[Lemma~4.7]{duSauGr}.

\begin{pps}\label{Digen} For each $i \in W'$, the domain of absolute convergence of the product $\prodq (1+\Zip)$ is 
 \begin{align*}\D_{i}:=\left\{(s_1,s_2)\in \C^2\mid\right.& \sum_{j \in M_i}\re\left(\dcj\right)>1-\sum_{j \in M_i}B_j,\\
 &\left.\re(\dcj)>-B_j,~\forall j \in M_i\cap W'\right\}.
 \end{align*}
\end{pps}
\begin{proof} If $j \in M_i\cap W'$, then each term $\frac{q^{\powerj}}{1-q^{\powerj}}$ converges absolutely if and only if $(s_1,s_2) \in \D_{j,1}$, for each $j \in M_i$. 
If $j \in M_i \setminus W'$, the corresponding term $\frac{q^{\powerj}}{1-q^{\powerj}}$ has no poles and converges on the whole~$\C^2$.

For $(s_1,s_2) \in \D_{j,1}$, the convergent sequence $((1-q^{\powerj})^{-1})$ is a decreasing sequence when $q$ increases, and hence it is bounded. 
The sequence $\left(\prod_{\theta=1}^{m-1}(1-q^{-\theta})^{-1}\right)$  
is also bounded when $q$ increases.
Therefore, to determine where the series $\sumq \Zip$ converges absolutely, it suffices to determine the domain of absolute convergence of the series
\[\sum_{\p \notin Q}(1-q^{-1})q^{-{d \choose 2}}c_i(\op)(q-1)^{|U_i|}q^{-\sum_{j \in M_i}(\astA_{1j}s_1+\astA_{2j}s_2+\astB_{j})}.\]
The Lang-Weil estimate of Lemma~\ref{langweil} guarantees that the series above converges absolutely if and only if so does the following series:
\[\sumq(1-q^{-1})q^{-{d \choose 2}-d_U}(q-1)^{|U_i|}q^{-\sum_{j\in M_i}(\astA_{1j}s_1+\astA_{2j}s_2+\astB_{j})},\]
which in turn converges absolutely for $(s_1,s_2) \in \C^2$ satisfying 
\[\re\left(\sum_{j \in M_i}\astA_{1j}s_1+\astA_{2j}s_2\right)>1-\sum_{j \in M_i}B_j+|U_i|-{d \choose 2}+d_{U_i}=1-\sum_{j \in M_i}\astB_j,\]
because of the identity $d_{U_i}={d \choose 2}-|U_i|$ and Proposition~\ref{Del}. It follows that the domain of absolute convergence of the series $\sumq \Zip$, 
and hence of the product $\prodq (1+\Zip)$ is $\D_i$, as desired. 
\end{proof}

If $i \in W'\cap [z]$, since $M_i=\{i\}$, the set $\D_i$ is given simply by
\begin{equation}\label{Di}
 \D_i=\{\setopt \mid \re(\dci)>1-\astB_i\}=\D_{i,0}.
\end{equation}
We now show that the domain of absolute convergence of $\astgc$ is given by an intersection of such sets.
\begin{cor} The product $\astgc$ converges on the domain  
\begin{equation}\label{holi}\hol=\holi:=\bigcap_{i \in [z]\cap W'}\D_i,\end{equation}
which is independent of the ring of integers $\ri$.
\end{cor}
\begin{proof}
It is clear that $\holi$ is independent of $\ri$, since so are the sets $\D_i$. 
Proposition~\ref{Digen} shows that $\astgc$ converges absolutely on $\bigcap_{i \in W'}\D_i$. We claim that $\bigcap_{i \in W'}\D_i=\bigcap_{i \in [z]\cap W'}\D_i$. 

Let $(s_1,s_2) \in \bigcap_{i \in W'\cap [z]}\D_i$. Given $k \in W'$
\begin{align*}\sum_{j \in M_k}\re(\astA_{1j}s_1+\astA_{2j}s_2) &=\sum_{j \in M_k \cap W'}\re(\astA_{1j}s_1+\astA_{2j}s_2)\\
&>\sum_{j \in M_{k} \cap W'}(1-\astB_j)\geq 1-\sum_{j \in M_k}\astB_j.
\end{align*}
The equality is justified by the fact that $(\astA_{1j},\astA_{2j})=(0,0)$ if and only if $ j \notin W'$, and the second inequality follows from the fact that 
$\astB_j\geq 0$ for all $j \in W'$. 
We have shown that $(s_1,s_2) \in \D_k$ for each $k \in W'$. Therefore, $\bigcap_{i \in [z]\cap W'}\D_i \subseteq \bigcap_{k \in W'}\D_k$.
\end{proof}

\subsubsection{Bad reduction}\label{badreduction}
For each $\p \in Q_1$, denote by $\mathscr{C}_\p$ the domain of convergence of the local factor $\bzfs_{\G(\ric)}(s_1,s_2)$. 
We now show that $\mathscr{C}_{\p} \supsetneq \holi$. A consequence is that 
\[\astbc=\prod_{\p \notin Q_2}\arel\]
converges absolutely on $\holi$ because $Q_1$ is a finite set.

Recall that the main terms of the bivariate zeta functions are given in~\eqref{mainterms} in terms of the $\p$-adic integrals $Z(\underline{s})$ of~\eqref{Vogen} at 
the points $(\asta_1s_1+\asta_2s_2+\astb)$. The poles of $Z(\asta_1s_1+\asta_2s_2+\astb)$, in turn, are the poles of functions 
$\Xi_{U,(d_{\kappa\iota})}^{N}(q, \mathbf{a}_1s_1+\mathbf{a}_2s_2+\mathbf{b})$, by Proposition~\ref{ppsDenef}(2).

The next proposition is analogous to \cite[Proposition~4.5]{AKOV13} and is proven in the same way.
\begin{pps}\label{akov45}
 Given $q$, $N \in \N$, a family of integers $(d_{\kappa\iota})$ for $\kappa \in [l]$ and $\iota \in J_{\kappa}$, and $\mathbf{a}_1$, $\mathbf{a}_2$, $\mathbf{b} \in \Z^l$, 
 the set of poles of $\Xi_{U,(d_{\kappa\iota})}^{N}(q, \mathbf{a}_1s_1+\mathbf{a}_2s_2+\mathbf{b})$ is independent of $q$, $N$ and $(d_{\kappa\iota})$, for all $U \subseteq T$.
\end{pps}

Since $\Xi_{U,(d_{\kappa\iota})}^{N}(q, \mathbf{a}_1s_1+\mathbf{a}_2s_2+\mathbf{b})$ may be written as 
\[q^{\sum_{u \in U}(N-1)\nu_u}\Xi_{U, (d_{\kappa\iota}^{1}+\sum_{u \in U}N_{u\kappa\iota}(N-1))}(q, \mathbf{a}_1s_1+\mathbf{a}_2s_2+\mathbf{b}),\]
it follows from Proposition~\ref{akov45} that the functions $\Xi_{U,(d_{\kappa\iota})}^{N}(q, \mathbf{a}_1s_1+\mathbf{a}_2s_2+\mathbf{b})$ and 
$\Xi_{U,(0)}^{1}(q, \mathbf{a}_1s_1+\mathbf{a}_2s_2+\mathbf{b})=\Xi_{U}(q, \mathbf{a}_1s_1+\mathbf{a}_2s_2+\mathbf{b})$
have the same poles.

In particular, the function $Z(\asta_1s_1+\asta_2s_2+\astb)$ converges absolutely on the domain $\bigcap_{i \in [z]\cap W'}\D_{j,1} \supsetneq \holi$ as in the good reduction case.  
This concludes the proof of Theorem~\ref{thmA}\eqref{parti}
\subsection{Meromorphic continuation}\label{meromorphic}
We start by showing that the bivariate function $\Gf$ admits meromorphic continuation to a domain $\tmer \supsetneq \holi$; see~Definition~\ref{meromorphy}. 
Recall that a \emph{domain} here means a connected open subset of~$\C^2$.

For each $i \in W'$, set $\textup{R}_i=\{j \in W' \mid \D_j=\D_i\}$, where $\D_i$ is the set defined in~\eqref{Di}. Set also 
\[\Rs=\left\{i \in W'\cap[z] \mid \bigcap_{j\in W'\setminus \textup{R}_i}\D_j\neq \holi \right\}.\]
In other words, $\Rs$ is the set of indices $i$ such that the boundary $\partial \D_i$ of $\D_i$ shares infinitely many points with the boundary $\partial\holi$ of $\holi$.

For each $\p \notin Q$, define 
\[\Vp=\prodr(1-c_i(\op)q^{-d_{U_i}}q^{\poweri}).\]

Recall from \eqref{goodZ} that $\Gf=\prodq(1+\sumw\Zip)$. Then, for $(s_1,s_2) \in \holi$, 
\[\Gf=\frac{\prodq(1+\sumw\Zip)\Vp}{\prodq\Vp},\]
provided that numerator and denominator converge.
In the following, we show that
\begin{enumerate}
 \item\label{Teili} The product $\prodq\Vp$ is meromorphic on a domain $\mero \supsetneq \holi$, which is independent of $\ri$, and
 \item\label{Teilii} The product $\prodq (1+\sumw\Zip)\Vp$ is meromorphic on a domain $\mert\supsetneq \holi$ which is independent of $\ri$.
\end{enumerate}

\subsubsection{Proof of~\eqref{Teili}}\label{pTeili}
For $i \in \Rs$, we define the following functions, which are analogous to the~$V_i(s)$ of~\cite[Section~4.2]{DuVo14}.
\[\Vi:=\prodq (1-c_i(\op)q^{d_{U_i}}q^{\poweri}).\]
It suffices to show that each $\Vi$ admits meromorphic continuation to $\D_{i,\Delta}$, for some $\Delta >0$.
Then, since $\Rs$ is finite, it will follow that 
\[\prodr \Vi = \prodq \Vp\]
admits meromorphic continuation to $\mero:=\bigcap_{i \in \Rs}\D_{i,\Delta}$. 

The following proposition is analogous to~\cite[Lemma~4.6]{duSauGr}.
\begin{pps}\label{Artin} For each $i\in W$ and $b \in I_{U_i}$, the function 
\[V_{b,i}(s_1,s_2)=\prodq (1-l_{\p}(F_{U_i,b})q^{\poweri})\]
converges absolutely on $\D_i$. 
Moreover, there exists~$\delta_i>0$ such that~$V_{b,i}(s_1,s_2)$ admits meromorphic continuation to 
$\D_{i,\delta_i}$.
\end{pps}
\begin{proof}
For each $i \in \Rs$ and $b \in I_{U_i}$, the convergence of $V_{b,i}(s_1,s_2)$ follows from the fact pointed out in the proof of~\cite[Lemma~4.6]{duSauGr} that 
$l_{\p}(F_{U_i,b})$ is bounded by the number of absolutely irreducible components of $F_{U_i,b}$. Then, for a sufficiently large $C>0$, the sum 
$\sumq l_{\p}(F_{U,b})q^{\poweri}$ is majored by $C\sumq q^{\poweri}$, which converges for $\re(\dci)>1-\astB_i$.

Let $L|K$ be a finite Galois extension and denote by~$S$ the finite set of prime ideals~$\p$ of~$\ri$ which are unramified and of 
the prime ideals $\p$ such that the reduction of $F_{U_i,b}$ mod~$\p$ is smooth. Denote by $\textup{Frob}_{\p}$ ($\p$ unramified) the conjugacy class in the Galois group of 
$L|K$ consisting of Frobenius elements. 
Given $a_1,a_2,b \in \R$ with $(a_1,a_2)\neq (0,0)$ and a representation~$\rho$ of the Galois Group of $L|K$, one can show that the Artin $L$-function 
\[L_{F_{U,b}}(a_1s_1+a_2s_2+b)=\prod_{\p} \det(1-\rho(\textup{Frob})_{\p}q^{-a_1s_1-a_2s_2-b})^{-1}\]
converges for $\re(a_1s_1+a_2s_2)>1-b$ and admits meromorphic continuation to the whole~$\C^2$, the same way that $L_{F_{U,b}}(s)$ does; see~\cite[Section~10 of Chap.VII]{Neukirch}. 
This is due essentially to the facts that, although we are considering two variables, the function~$L_{F_U,b}(a_1s_1+a_2s_2+b)$ is being taken over values on~$\C$ given by the 
entire function $\omega: \C^2 \to \C$ defined by $\omega(s_1,s_2)=a_1s_1+a_2s_2+b$.

In particular, the second part of this proposition follows from similar arguments as the ones of~\cite[Lemma~4.6]{duSauGr}.
\end{proof}

For each $i \in \Rs$, define 
\begin{equation}\label{seriesvi}\Vit=\prod_{b \in I_{U_i}}\prodq(1-l_{\p}(F_{U_i,b})q^{\poweri})=\prod_{b \in I_{U_i}}V_{b,i}(s_1,s_2).\end{equation}
Since~$I_{U_i}$ is finite, Proposition~\ref{Artin} assures that $\Vit$ converges on $\D_i$ and admits meromorphic continuation to~$\D_{i,\delta_i}$ for some~$\delta_i>0$.
Moreover, for each $b \in I_{U_i}$ the sum $\sumq l_{\p}(F_{U_i,b})q^{\poweri}$ converges absolutely on $\D_i$ and admits meromorphic continuation to~$\D_{i,\delta_i}$.
Consequently, the sum
\begin{align*}&\sumq \left|\left(c_i(\op)-\sum_{b \in I_{U_i}}l_{p}(F_{U,b})q^{d_{U_i}}\right)\cdot\left(q^{-d_{U_i}}q^{\poweri}\right)\right|
\end{align*}
converges on $\D_{i,\Delta_i}$ for some $\Delta_i > 0$, by the Lang-Weil estimate of Lemma~\ref{langweil}. 
It follows that $\Vi$ is a meromorphic function on $\D_{i,\min\{\delta_i,\Delta_i\}}$, and therefore 
$\prodq \Vi(s_1,s_2)$ is meromorphic on $\bigcap_{i \in \Rs}\D_{i,\Delta}$, for $\Delta=\min\{\delta_i, \Delta_i\mid i \in \Rs\}$.
\subsubsection{Proof of~\eqref{Teilii}}\label{pTeilii}
We now introduce some notation which will be convenient while proving that the product $\prodq(1+\sumw \Zip)\Vp$ is meromorphic on a domain 
$\mert\supsetneq \holi$ which is independent of $\ri$.
\begin{dfn}
 Given families $(f_{\p}(s_1,s_2))_{\p\notin Q}$ and $(g_{\p}(s_1,s_2))_{\p\notin Q}$ of bivariate complex functions and a domain $\mathcal{D}$, we write 
 \[\prodq f_{\p}\equiv_{\mathcal{D}}\prodq g_{\p}\]
 to indicate that $\sumq (f_{\p}(s_1,s_2)-g_{\p}(s_1,s_2))$ is absolutely convergent on $\mathcal{D}$.
\end{dfn}

This is a modification of the relations $\equiv$ of \cite[Section~4]{duSauGr} and $\equiv_{\Delta}$ of \cite[Definition~4.4]{DuVo14}. 
The following Lemmata state convenient properties of $\equiv_{\mathcal{D}}$ which will be used. 
\begin{lem}\label{equiv}
 Let $(f_{\p}(s_1,s_2))$, $(g_{\p}(s_1,s_2))$, and $(h_{\p}(s_1,s_2))$ be families of of bivariate complex functions indexed by $\p \notin Q$, and let~$\mathcal{D}$
 and~$\mathcal{D}'$ be domains of~$\C^2$. If $\prodq f_{\p}\equiv_{\mathcal{D}}\prodq g_{\p}$ and 
 $\prodq g_{\p} \equiv_{\mathcal{D}'}\prodq h_{\p}$, then 
 \[\prodq f_{\p} \equiv_{\mathcal{D}\cap \mathcal{D}'}\prodq h_{\p}.\]
 In particular, if $\prodq g_{\p}(s_1,s_2)$ converges absolutely on the domain $\mathcal{D}'$ and 
 $\prodq f_{\p}\equiv_{\mathcal{D}}\prodq g_{\p}$, then $\prodq f_{\p}(s_1,s_2)$ converges absolutely on the domain 
 $\mathcal{D}\cap\mathcal{D}'$.
\end{lem}
\begin{proof}
 The first claim follows from the fact that
 \[\sumq|f_{\p}(s_1,s_2)-h_{\p}(s_1,s_2)| \leq \sumq(|f_{\p}(s_1,s_2)-g_{\p}(s_1,s_2)|+|g_{\p}(s_1,s_2)-h_{\p}(s_1,s_2)|).\]
 By definition, $\prodq g_{\p}(s_1,s_2)$ being absolutely convergent on~$\mathcal{D}'$ is equivalent to $\prodq g_{\p}\equiv_{\mathcal{D}'}1$.
 The second claim then follows from the first part of Lemma~\ref{equiv}.
\end{proof}
\begin{lem}\label{equivprod}
 Let $(f_{\p}(s_1,s_2))$, $(g_{\p}(s_1,s_2))$, and $(X_{\p}(s_1,s_2))$ be families of bivariate complex functions indexed by $\p \notin Q$, and let~$\mathcal{D}$ 
 and~$\mathcal{D}'$ be domains of~$\C^2$. If $\prodq f_{\p}\equiv_{\mathcal{D}}\prodq g_{\p}$ and $(X_{\p}(s_1,s_2))$ is bounded on~$\mathcal{D}'$, then 
 \[\prodq f_{\p}X_{\p}\equiv_{\mathcal{D}\cap \mathcal{D}'}\prodq g_{\p}X_{\p}.\]
\end{lem}
\begin{proof}
 This is clear, as the partial sums
 $\sum |f_{\p}(s_1,s_2)-g_{\p}(s_1,s_2)||X_{\p}(s_1,s_2)|$ are bounded on $\mathcal{D}\cap\mathcal{D}'$.
\end{proof}

In the following, we denote $\D_{\Rs,\delta}=\bigcap_{i \in \Rs}\D_{i,\delta}$ for each~$\delta>0$.
\begin{pps}\label{D1}
There exists a domain~$\mathcal{D}_1$ which is independent of~$\ri$ satisfying the following condition: 
for each $\delta>0$ the intersection $\mathcal{D}_1\cap \D_{\Rs,\delta}$ is a domain strictly containing~$\holi$ and such that 
\begin{equation}\label{eqD1}\prodq \left(1+\sumw \Zips \right)\Vps\equiv_{\mathcal{D}_1} \prodq \left(1 + \sumr \Zips \right)\Vps.\end{equation}
\end{pps}
\begin{proof}
 The domain $\mathcal{D}'_{1}:=\bigcap_{i \in W'\setminus \Rs}\D_i$ strictly contains~$\holi$, by choice of~$\Rs$, and is independent of~$\ri$, 
 since so are  the domains~$\D_i$, for all $i \in W'$.
 
 The domain $\mathcal{D}'_{1}$ has the property that, for each $\delta>0$, the intersection $\mathcal{D}'_{1}\cap \D_{\Rs,\delta}$ is an open 
 domain strictly containing $\holi$. In fact, if $\mathcal{D}'_{1}\cap \D_{\Rs,\delta}=\holi$, then since $\D_{\Rs,\delta}$ is a translation of $\holi$ which strictly contains 
 the latter,  we must have $\mathcal{D}'_{1}=\holi$. 
 
 The definition of $\equiv_{\mathcal{D}_1}$ yields
 \[\prodq \left(1+\sumw \Zips\right)\equiv_{\mathcal{D}_1} \prodq \left(1 + \sumr \Zips \right).\]
 Since the sequence $(\Vp)_{\p\notin Q}$ is positive monotonically non-increasing on $\mathcal{D}_{V_{\p}}:=\bigcap_{i\in \Rs}\D_{i,d_{U_i}+1}$,
 Lemma~\ref{equivprod} assures that~\eqref{eqD1} holds for $\mathcal{D}_1:=\mathcal{D}'_{1}\cap \D_{V_{\p}}$.
 
 Clearly, given $\gamma$, $\gamma' >0$, the intersection $\D_{\Rs, \gamma}\cap \D_{\Rs,\gamma'}$ is $\D_{i, \min\{\gamma,\gamma'\}}$. Thus, for each $\delta>0$, and for 
 $\gamma:=\min\{d_{U_i} \mid i \in \Rs\}$, 
 \[\mathcal{D}_1\cap \D_{\Rs, \delta} \supseteq \mathcal{D}'_{1} \cap \D_{\Rs, \min\{\delta,\gamma+1\}} \supsetneq \holi.\qedhere\]
\end{proof}

We now use Lemma~\ref{equiv} and the auxiliary functions 
\[\Zipb=c_i(\op)q^{-d_{U_i}}q^{\poweri}, ~i\in \Rs,\]
to show that there exists $\delta>0$ such that
\begin{equation}\label{equivD1}\prodq\left(1+\sumr \Zips\right)\Vps\equiv_{\D_{\Rs,\delta}}1.\end{equation}
Then \eqref{equivD1} and Proposition~\ref{D1} together imply 
\[\prodq\left(1+\sumw \Zips\right)\Vp\equiv_{\mert}1,\]
where $\mert:=\mathcal{D}_1 \cap \D_{\Rs,\delta}$, which is independent of the ring of integers $\ri$.
In preparation for this, we need three lemmata. 
\begin{lem}\label{D2}
 There exist $\delta_2>0$ and a domain $\mathcal{D}_{2}\supseteq \D_{\Rs,\delta_2}$ such that
 \[\prodq \Vps \equiv_{\mathcal{D}_2} \prodq \left(1- \sumr\Zipbs\right).\] 
\end{lem}
\begin{proof}
We first notice that 
 \begin{align}
  &\sumq|\Vp-(1-\sumr\Zipb)|\nonumber \\
  &=\sumq\left|\prodr\left(1-\Zipb\right)-\left(1-\sumr \Zipb\right)\right|\nonumber\\
  \label{d2} &=\sumq \left| \sum_{l=2}^{|\Rs|}\sum_{\substack{I \subseteq \Rs\\ |I|=l}}(-1)^l\prod_{i \in I}\Zipb\right|.
 \end{align}
By applying successively the Lang-Weil estimate of Lemma~\ref{langweil} to~\eqref{d2}, we obtain that $\sumq|\Vp-(1-\sumr\Zipb)|$ converges if and only if the series 
\[\sumq \left|\sum_{l=2}^{|\Rs|}\sum_{\substack{I \subseteq \Rs\\ |I|=l}}(-1)^lq^{-\sum_{i \in I}(\astA_{1i}s_1+\astA_{2i}s_2+\astB_i)}\right|\]
converges, which in turn converges on the domain 
\[\mathcal{D}_{2}:=\left\{\setopt \mid \sum_{i \in I}\re(\dci)>1-\sum_{i \in I}B_{i},~I\subseteq \Rs \text{ with }|I|\geq 2\right\}.\]

Finally, if $(s_1,s_2) \in \D_{\Rs, \frac{1}{2}}$, then for each $I\subseteq \Rs$ with $|I|\geq 2$,
\[\sum_{i \in I}\re(\dci)>\sum_{i \in I}\left(\frac{1}{2}-\astB_i\right)\geq 1-\sum_{i \in I}\astB_i,\]
that is, $\D_{\Rs,\frac{1}{2}}\subseteq \mathcal{D}_2$.
\end{proof}

\begin{lem}\label{D3}
 There exist $\delta_3>0$ and a domain $\mathcal{D}_{3}\supseteq \D_{\Rs,\delta_3}$ such that
 \[\prodq \left(1+\sumr \Zips\right) \equiv_{\mathcal{D}_3} \prodq \left(1+ \sumr\Zipbs\right).\] 
\end{lem}
\begin{proof}
 For each $\p \notin Q$ and $i\in\Rs$, denote 
 \begin{align*}
 &\mathcal{S}_{\p,i}(s_1,s_2)=(1-q^{-1})^dq^{-{d \choose 2}}(q-1)^{|U_i|}q^{d_{U_i}}-(1-q^{\poweri})\prod_{\theta=1}^{d-1}(1-q^{-\theta}).
 \end{align*}
 For each $i \in \Rs$ the sequences $(\prod_{\theta=1}^{d-1}(1-q^{-\theta})^{-1})$ and $((1-q^{\poweri})^{-1})$ are positive and monotonically non-increasing for $\re(\dci)>-\astB_i$ 
 when $q$ increases. Thus, if the series 
 \[\sumq \sumr \left| \mathcal{S}_{\p,i}(s_1,s_2)c_i(\op)q^{-d_{U_i}}q^{\poweri}\right|\]
converges absolutely on $\mathcal{D}_{3}$, then the series
\begin{align*}
 &\sumq \sumr \left|\Zip-\Zipb\right|\\
&=\sumq \sumr \frac{|\mathcal{S}_{\p,i}(s_1,s_2)c_i(\op)q^{-d_{U_i}}q^{\poweri}|}{|(\prod_{\theta=1}^{d-1}(1-q^{-\theta}))(1-q^{\poweri})|}
 \end{align*}
also converges absolutely on $\mathcal{D}_3\cap \D_{\Rs,1}$. 

The claim of Lemma~\ref{D3} then follows from the fact that the series 
\[\sumq\sumr \left(c_i(\op)q^{-2\astA_{1i}s_1-2\astA_{2i}s_2-2\astB_i}\prod_{\theta=1}^{d-1}(1-q^{-\theta})\right)\]
converges absolutely on $\D_{\Rs,\frac{1}{2}}$, because of the Lang-Weil estimate of Lemma~\ref{langweil} and Proposition~\ref{Del}.
%
%
\end{proof}
%
%
\begin{lem}\label{D4}
 The product 
\[\prodq \left(\left(1+\sumr \Zipb\right)\left(1-\sumr \Zipb\right)\right)\] 
converges absolutely on the domain $\D_{\Rs,\frac{1}{2}}$.
\end{lem}
\begin{proof}
 Let us show that \[\prodq \left(1+\sumr \Zipb\right)\left(1-\sumr \Zipb\right)\equiv_{\D_{\Rs,\frac{1}{2}}} 1.\]
In fact,
 \begin{align*}
 &\sumq \left| 1-(1+\sumr \Zipb)(1 -\sumr \Zipb)\right|\\
 &= \sumq \sumr \sum_{j \in \Rs} \left| \Zipb \overline{\bzfs_{j,\p}}(s_1,s_2)\right|\\
 &= \sumq \sumr \sum_{j \in \Rs} \left| c_i(\op)c_j(\op)q^{-d_{U_i}-d_{U_j}}q^{-(\astA_{1i}+\astA_{1j})s_1-(\astA_{2i}+\astA_{2j})s_2-(\astB_{i}+\astB_{j})}\right|,
\end{align*}
which, by Lemma~\ref{langweil}, converges if and only if the following series converges:
\[\sumq \sumr \sum_{j \in \Rs} \left| q^{-(\astA_{1i}+\astA_{1j})s_1-(\astA_{2i}+\astA_{2j})s_2-(\astB_{i}+\astB_{j})}\right|.\]
Proposition~\ref{Del} assures that the latter series converges on 
\[\mathcal{D}_4:=\{\setopt \mid \re((\astA_{1i}+\astA_{1j})s_1+(\astA_{2i}+\astA_{2j})s_2)>1-\astB_i-\astB_j,~i,j \in \R\}.\]
In particular, if we choose $i=j$ in $\Rs$, we see that for each $(s_1,s_2)\in \mathcal{D}_4$,
\[\re(\dci)>\frac{1-2\astB_i}{2}=1-\astB_i-\frac{1}{2}.\]
In other words, $\mathcal{D}_4 \subseteq \D_{\Rs,\frac{1}{2}}$. The equality $\mathcal{D}_4 = \D_{\Rs,\frac{1}{2}}$ holds, since $(s_1,s_2)\in \D_{\Rs,\frac{1}{2}}$ implies
\[\re((\astA_{1i}+\astA_{1j})s_1+(\astA_{2i}+\astA_{2j})s_2)>\frac{1-2\astB_i}{2}+\frac{1-2\astB_j}{2}=1-B_i-B_j.\qedhere\]
\end{proof}

There is $\delta>0$ such that the domains $\mathcal{D}_2$ and $\mathcal{D}_3$ of Lemmata~\ref{D2} and \ref{D3} satisfy
\[\mathcal{D}_2 \cap \mathcal{D}_3 \cap \D_{\Rs,\frac{1}{2}}\supseteq \D_{\Rs,\delta_2}\cap \D_{\Rs,\delta_3}\cap \D_{\Rs,\frac{1}{2}}=\D_{\Rs,\delta}.\]
It then follows from Lemmata~\ref{equiv}, \ref{D2}, \ref{D3}, and~\ref{D4} that 
\begin{align*}
\prodq(1+\sumr \Zips)\Vps &\equiv_{\D_{\Rs,\delta}} \prodq(1+\sumr \Zips)(1- \sumr\Zipbs) \\
			  & \equiv_{\D_{\Rs,\delta}} \prodq(1+\sumr \Zipbs)(1- \sumr\Zipbs) \equiv_{\D_{\Rs,\delta}} 1,\\
\end{align*}
which confirms~\eqref{equivD1}.
\subsubsection{Proof of Theorem~\ref{thmA}\eqref{partii}}
It follows from the results of Sections~\ref{pTeili} and~\ref{pTeilii} that $\Gf$ is meromorphic on the domain $\mero \cap \mert$, which is independent of~$\ri$.
Moreover, $\mero=\D_{\Rs,\Delta}$ for some $\Delta >0$ and the intersection of~$\mert$ with a domain of the form~$\D_{\Rs,\delta}$ with $\delta>0$ is 
a domain strictly containing~$\holi$. 

In Section~\ref{badreduction}, we have shown that, for $\p \in Q_1$, the domain of convergence $\mathscr{C}_\p$ of $\arel$ is a domain of the form 
$\bigcap_{i \in [z]\cap W'}\D_{i,\delta}$. 
Denote by $\mathscr{C}_{Q_1}$ the intersection of all $\mathscr{C}_{\p}$ with $\p \in Q_1$. 

Since the function \[\prod_{\p \notin Q_2}\arel\] is meromorphic on $\meri=\mer:=\mero \cap \mert \cap \mathscr{C}_{Q_2}$, it is left to show that~$\meri$ is a domain strictly 
containing~$\holi$.

In fact, for each $i\in [z]\cap W'$ the domain~$\D_{i,\delta}$ is a translation of the domain~$\D_i$. Thus, $\Rs$ is also the set of all indices $i \in [z]\cap W'$ such that 
the boundary $\partial \D_{i,\delta}$ shares infinitely many points with the boundary~$\partial \left(\bigcap_{i \in [z]\cap W'}\D_{i,\delta}\right)$. 
In other words, $\bigcap_{i \in [z]\cap W'}\D_{i,\delta}=\bigcap_{i \in \Rs}\D_{i,\delta}= \D_{\Rs,\delta}$. 
Therefore, the domains of convergence $\mathscr{C}_\p$ for $\p \in Q_1$ are domains of the form $\D_{\Rs,\delta}$ with $\delta>0$, and hence $\mathscr{C}_{Q_1}=\D_{\Rs,\gamma}$ 
for some $\gamma>0$. 

This concludes the proof of Theorem~\ref{thmA}\eqref{partii}.
\section*{Acknowledgements}
This paper is part of my PhD thesis. 
I am grateful to my advisor, Christopher Voll, for his constant support and for helpful discussions. 
I would like to thank Yuri Santos Rego for his support and for his comments on a draft of this paper. 
I also gratefully acknowledge financial support from the DAAD for this work.

 \def\cprime{$'$} \def\cprime{$'$}
 \providecommand{\bysame}{\leavevmode\hbox to3em{\hrulefill}\thinspace}
 \providecommand{\MR}{\relax\ifhmode\unskip\space\fi MR }

\printbibliography

\end{document}